\documentclass[12pt]{article}
\usepackage{cmap}
\usepackage[T2A]{fontenc}
\usepackage[utf8]{inputenc} % [utf8][cp866][cp1251]
\usepackage[russian]{babel}
\usepackage{amsmath, amssymb, latexsym}
\usepackage{tikz}
\usepackage{ifthen}

\makeatletter
\def\@seccntformat#1{\csname the#1\endcsname.\ } % точка после номера раздела
\makeatother

\newif\ifNoRemark
\def\addtheorem#1#2#3#4{
\ifthenelse{\equal{#2}{}}{}%
{\ifthenelse{\expandafter\isundefined\csname the#2\endcsname}{\newcounter{#2}}{}}
\newenvironment{#1}[1][\global\NoRemarktrue]% No Remark by default
{\par\addvspace{2mm plus 0.5mm minus 0.2mm}\noindent % new paragraph without indent
{\bf #3}\ifthenelse{\equal{#2}{}}{}%
{\refstepcounter{#2}{\bf ~\csname the#2\endcsname}}%
{\bf \vphantom{##1}\ifNoRemark.\ \else\ (##1).\fi}\begingroup #4}%
{\endgroup\par\addvspace{1mm plus 0.5mm minus 0.2mm}\global\NoRemarkfalse}
\expandafter\newcommand\csname b#1\endcsname{\begin{#1}}
\expandafter\newcommand\csname e#1\endcsname{\end{#1}}
}

\addtheorem{theorem}{thrm}{Теорема}{\sl}
\addtheorem{lemma}{lmm}{Лемма}{\sl}
\addtheorem{predl}{prpstn}{Предложение}{\sl}
\addtheorem{corollary}{crllr}{Следствие}{\sl}
\addtheorem{remark}{rmrk}{Remark}{}

 \newenvironment{proof}[1][\hspace{-1.0ex}]%
  {\par\addvspace{1mm}{\sc Доказательство\hspace{1.0ex}{#1}.} }%
  {\quad$\blacktriangle$\par\addvspace{1mm}}

\def\extshn#1#2{#2^{[#1]}}
\date{УДК 519.143}
\title{Об одном признаке свитчинговой разделимости графов по модулю $q$%
\thanks{Исследование выполнено за счет гранта Российского научного фонда (проект №14-11-00555)}
}
\author{Беспалов Е. А.%
\thanks{Беспалов Евгений Андреевич,
ММФ НГУ, Новосибирск 630090 Пирогова 2,
ИМ СО РАН, Новосибирск 630090 пр. Академика Коптюга 4,
bespalovpes@mail.ru}%
, Кротов Д. С.%
\thanks{Кротов Денис Станиславович,
ИМ СО РАН, Новосибирск 630090 пр. Академика Коптюга 4,
ММФ НГУ, Новосибирск 630090 Пирогова 2,
krotov@math.nsc.ru}
}
\begin{document}
\maketitle
\begin{abstract}
 We consider the graphs whose edges are marked by the integers (weights) from $0$ to $q-1$ (zero corresponds to no-edge). Such graph is called additive if its vertices can be marked in such a way that the weight of every edge is equal to the modulo-$q$ sum of weights of the two incident vertices. By a switching of a graph we mean the modulo-$q$ sum of the graph with some additive graph on the same vertex set. A graph with $n$ vertices is called switching separable if some of its switchings does not have a connected component of order $n$ or $n-1$. We consider the following test for the switching separability: if removing any vertex of a graph $G$ results in a switching separable graph, then $G$ is switching separable itself. We prove this test for odd $q$ and characterize the exceptions when $q$ is even. We establish a connection between the switching separability of a graph and the reducibility of $(n-1)$-ary quasigroups constructed from this graph.
 
 Рассматриваются графы, 
 ребра которых помечены числами (весами) 
 от $0$ до $q-1$ 
 (ноль соответствует отсутствию ребра).
Граф называется аддитивным, 
если вершины можно пометить таким образом,
что вес ребра равен 
сумме меток двух его вершин по модулю $q$.
Свитчингом данного графа 
называется его сумма по модулю $q$ 
с некоторым аддитивным графом 
на том же множестве вершин.
Граф на $n$ вершинах называется 
свитчингово разделимым, если 
некоторый его свитчинг не имеет 
компонент связности мощности больше $n-2$.
Мы рассматриваем следующий признак 
свитчинговой разделимости:
если удаление любой вершины графа $G$ 
приводит к свитчингово разделимому графу,
то и сам граф $G$ свитчингово разделим. 
Мы доказываем этот признак для нечетного $q$ 
и характеризуем множество исключений 
для четного $q$.
Устанавливается связь между 
свитчинговой разделимостью графа и
разделимостью $n$-арной квазигруппы, 
построенной по этому графу.
\end{abstract}
%=================================================================
 \section{Введение}\label{s:intro}
Мы изучаем связь свитчинговой разделимости графов по модулю $q$ с разделимостью подграфов.
Основной рассматриваемый вопрос --- существует ли неразделимый граф, 
у которого удаление любой вершины приводит к разделимому графу?
Назовем такие графы критическими.
Получен исчерпывающий ответ на этот вопрос: 
класс критических графов удалось полностью охарактеризовать,
они существуют при любом четном $q$ 
и любом нечетном числе вершин, начиная с $5$.

Этот вопрос оказывается важным 
с точки зрения изучения другого класса 
комбинаторных объектов 
--- $n$-арных квазигрупп, 
в комбинаторике известных также 
как латинские гиперкубы.
Напомним, $n$-арной квазигруппой 
конечного порядка $q$ 
называется $n$-местная операция
$\Sigma^n \to \Sigma$, 
обратимая по каждому аргументу, 
где $\Sigma$ --- некоторое множество мощности $q$
(носитель квазигруппы). 
$n$-Арная квазигруппа называется разделимой, 
если она представима 
в виде бесповторной суперпозиции 
квазигрупп меньшей арности.
Для $n$-арных квазигрупп 
можно сформулировать аналогичный вопрос:
существует ли такая неразделимая 
$n$-арная квазигруппа (назовем ее критической), 
что подстановка любой константы на место любого аргумента
приводит к разделимой $n$-арной квазигруппе? 
Предположим, что мы изучаем некоторый класс квазигрупп, 
для которого ответ на этот вопрос --- ``нет'',
и мы пытаемся доказать некоторую гипотезу о строении любой неразделимой квазигруппы
(самый простой вариант гипотезы --- что неразделимых вообще нет, 
пример подобных рассуждений см. в \cite[раздел~6]{KroPot:retr}). 
Тогда согласно индукционному предположению мы можем считать,
что фиксацией некоторого аргумента из исходной $n$-арной квазигруппы мы получаем
$(n-1)$-арную квазигруппу, для которой гипотеза верна 
(если мыслить $n$-арную квазигруппу 
как $n$-мерную таблицу значений, 
то нам известно,
что гипотеза верна для некоторого слоя таблицы).
Подобный ход рассуждений 
был успешно применен при характеризации 
$n$-арных квазигрупп порядка $4$ \cite{KroPot:4}, 
хотя именно наличие критических квазигрупп 
сильно усложнило доказательство,
вынуждая использовать индукционные соображения глубины два 
и более слабую версию признака разделимости.
В настоящее время известно,
что не существует критических 
$n$-арных квазигрупп простого порядка 
\cite{KroPot:retr} и существуют критические 
$n$-арные квазигруппы любого порядка,
кратного $4$ \cite{Kro:n-2}, при любом четном $n\ge 4$.
Последние строятся при помощи критических графов,
где разделимость рассматривается по модулю $2$ 
(к слову, в случае $q=2$ свитчинги графа хорошо известны 
в литературе как свитчинги Зейделя \cite{vLintSeid:66},
а свитчинговые классы эквивалентны 
так называемым два-графам \cite{Spence:2graphs}).

Достаточно естественным является вопрос 
о возможности построения критических квазигрупп 
на основе критических графов при $q>2$. 
В качестве промежуточного этапа 
в построении квазигрупп из графов, 
в случае простого $q$, 
используются квадратичные функции 
над полем $\mathbb{Z}_q$ из $q$ элементов,
а порядок полученных квазигрупп равен $q^2$, 
эти построения описаны в разделе~\ref{s:gbq}
(таким образом, минимальный порядок квазигрупп, 
для которого ставится вопрос, равен $9$).

В настоящей работе описываются все критические графы.
В частности, установлено, что в случае нечетного $q$ 
таких графов нет, и получить критические квазигруппы
нечетного порядка описанным способом невозможно.
Это позволяет предположить, что критических квазигрупп
нечетного порядка также нет, 
и более того, общий ход рассуждений
для доказательства этого факта может быть 
частично заимствован из доказательства для графов.
С другой стороны, квазигруппы, построенные по графам,
являются очень частным случаем, 
и разумно предположить, что 
гипотетическое доказательство для квазигрупп 
будет значительно сложнее.

В разделе \ref{s:aux} формулируется и доказывается
основной результат настоящей работы (теорема~\ref{t:main}).
Мы определяем класс графов $G_{n,\gamma}$
и показываем, что любой критический граф 
является свитчингом графа из этого класса.
Важным следствием является то, что при нечетных $q$ критических графов нет.
Частный случай $q=2$ был рассмотрен ранее 
в работах \cite{Kro2010ru:swi}, \cite{Besp:2014}.

В разделе \ref{s:gbq} показана связь 
свитчинговой разделимости графов по модулю $q$
с разделимостью $n$-арных квазигрупп порядка $q^2$, 
построенных по этим графам
(предложение~\ref{l:f-to-q}), и 
(в качестве промежуточного звена)
с разделимостью частичных функций над $\mathbb{Z}_q$, 
определенных на наборах 
с фиксированной суммой (предложение~\ref{l:bol-sep}).

Завершая введение, 
хочется обратить внимание читателя 
на работу \cite{Zaslavsky:2012},
где также связность графов используются 
для получения результатов о разделимости 
(точнее о приводимости --- 
частном случае разделимости) 
$n$-арных квазигрупп. 
Однако как подходы, 
так и полученные результаты
почти не пересекаются 
с содержанием настоящей статьи.

%=================================================================
\section{Разделимость графов}\label{s:aux}

Будем рассматривать неориентированные графы, 
ребра которых помечены 
элементами из множества 
$\{1,\ldots,q-1\}$, 
$q\geq 2$ --- натуральное, 
которые будем называть \emph{весом} ребра
(вес можно также трактовать как кратность). 
Метку $0$ будем ассоциировать 
с отсутствием ребра.
Таким образом реберно 
помеченный граф удобно 
представлять парой $(V,E)$,
где $E:V^2 \to \{0,1,\ldots,q-1\}$
--- симметричное отображение, 
равное нулю везде на диагонали 
$\{(v,v) | v\in V\}$.
Под \emph{подграфом} 
графа $G=(V,E)$ 
будем подразумевать подграф 
$G_W=(W,I)$, 
порожденный множеством
вершин $W\subset V$ 
и унаследовавший от $G$ 
веса ребер: $I(v,w)=E(v,w)$, 
для любых $v,w\in W$.
%Также под сложением будем подразумевать сложение по модулю $q$. 
Результатом сложения 
двух графов $G_1$  и $G_2$ 
с общим множеством вершин
будет граф $G$ 
на том же множестве вершин, 
определенный следующим образом: 
вес любого ребра 
графа $G$ равен 
сумме по модулю $q$ 
весов соответствующих ребер 
в графах $G_1$  и $G_2$.
Граф будем называть \emph{аддитивным}, 
если каждую его вершину 
можно пометить числами 
от $0$ до $q-1$ таким образом, 
для вес каждого ребра 
будет равен 
сумме по модулю $q$ 
меток двух вершин ребра. 
Далее определим 
\emph{свитчинг} графа $G$, 
как результат сложения графа $G$ 
с некотором аддитивным графом. 
Множество вершин $W$ 
графа $G$ назовем \emph{отделимым}, 
если некоторый свитчинг графа $G$ 
не содержит ребер 
между $W$ и $V\backslash W$.
Легко видеть, 
что любое множество вершин 
мощности $0$, $1$, $n-1$ или $n$ 
в графе порядка $n$ 
является отделимым.
Любые другие отделимые множества 
назовем \emph{нетривиальными}.
Граф $G=(V,E)$ назовем 
\emph{свитчингово разделимым} 
(далее в тексте --- 
просто \emph{разделимым}), 
если существует нетривиальное 
отделимое множество его вершин. 

\begin{lemma}\label{l:close}
Множество аддитивных графов замкнуто 
относительно сложения.
\end{lemma}
\begin{proof}
Утверждение следует 
прямо из определения: 
в качестве метки каждой 
вершины результирующего графа
можно взять сумму меток 
этой вершины в графах-слагаемых.
\end{proof}

Таким образом, отношение 
``граф $G$ --- свитчинг графа $H$'' 
является эквивалентностью.
Если какое-то множество отделимо 
в графе (или в некотором подграфе), 
то оно отделимо 
и в каждом его свитчинге 
(в соответствующем подграфе), 
поэтому в вопросах разделимости
мы без потери общности 
можем рассматривать 
наиболее удобный
свитчинг данного графа. 
В частности,
мы всегда можем 
считать некоторую одну
данную вершину изолированной.
Далее под словами 
``изолируем вершину $o$'' 
будем подразумевать
рассмотрение без потери общности
свитчинга исходного графа, 
у которого вершина $o$ изолирована.

\begin{lemma}\label{l:sepset}
Множество вершин $W$ 
графа $G=(V,E)$ отделимо 
тогда и только тогда, 
когда существуют 
$W_0,\ldots,W_{q-1}$,
$V_0,\ldots,V_{q-1}$ 
такие, что 
$W=W_0\cup W_1\cup\ldots \cup W_{q-1}$, 
$V\backslash W=V_0\cup V_1\cup\ldots\cup V_{q-1}$ 
и любое ребро, 
соединяющее вершины 
из множеств $W_i$ и $V_j$ 
имеет вес $i+j$.
\end{lemma}
\begin{proof}
Сопоставив каждой вершине 
из $W_i$ или $V_i$ метку $q-i$, 
мы породим некоторый 
аддитивный граф. 
Очевидно, что его сумма 
с исходным графом 
не имеет ребер, 
соединяющих $W$ 
и $V\backslash W$.
\end{proof}

\begin{corollary}\label{c:rn}
Пусть множество вершин 
$W$ графа $G=(V,E)$, 
имеющее изолированную 
вершину, отделимо. 
Тогда любая вершина из $W$ 
соединена со всеми вершинами 
из $V\backslash W$ 
ребрами одного веса.
\end{corollary}
\begin{proof}
Допустим,
в разбиении множества $V\backslash W$ 
из леммы~\ref{l:sepset} 
есть хотя бы два непустых множества. 
Но тогда вершины из разных множеств 
будут соединены с изолированной вершиной 
ребрами разных весов, что неверно. 
Поэтому разбиение $V\backslash W$ 
состоит только из одного 
непустого множества,
и любая вершина из $W$ 
соединена со всеми вершинами 
из $V\backslash W$
ребрами одного и того же веса.
\end{proof}
\begin{corollary}\label{c:rd}
Пусть любая вершина 
из множества $V\backslash W$ 
соединена со всеми вершинами $W$ 
ребрами одного и того же веса. 
Тогда $W$ отделимо.
\end{corollary}
\begin{proof}
Обозначим $W_0=W$, 
$W_1=\ldots=W_{q-1}=\emptyset$. 
Множество вершин из $V\backslash W$, 
соединенных с вершинами из $W$ 
ребрами веса $i$ обозначим через $V_i$. 
Тогда по лемме~\ref{l:sepset} 
$W$ отделимо.
\end{proof}
\begin{corollary}\label{c:zm}
Пусть граф $G$ разделим 
и имеет изолированую вершину $o$. 
Тогда для любых 
$i,j\in\{0,1,\ldots ,q-1\}$ 
если в графе $G$ поменять 
веса $i$ и $j$ местами, 
то получившийся граф 
также будет разделим. 
\end{corollary}
\begin{proof}
Пусть $W$ --- нетривиальное отделимое множество, не содержащее вершину $o$. По следствию~\ref{c:rn} любая вершина из $V\backslash W$ соединена с всеми вершинами из $W$ ребрами одного и того же веса. Но если поменять местами веса $i$ и $j$ в графе $G$, то полученный граф будет раделим по следствию~\ref{c:rd}.
\end{proof}

\begin{lemma}\label{l:ns}
Пусть $D$ --- разделимый граф порядка больше $4$ и $d$ --- некоторая его вершина. Граф $D\backslash\{d\}$ неразделим тогда и только тогда, когда вершина $d$ в графе $D$ принадлежит единственному нетривиальному отделимому множеству вершин $W=\{d,e\}$, для некоторой вершины $e$.
\end{lemma}
\begin{proof}
Пусть $D=(V,E)$ --- разделимый граф, $|V|\ge 5$, $D\backslash\{d\}$ --- его неразделимый подграф и $W$ --- нетривиальное отделимое множество, содержащее $d$. Сначала предположим, что $|W\backslash\{d\}|\ge 2$ . Тогда $W\backslash\{d\}$ и $V\backslash W\backslash\{d\}$ содержат по крайней мере по две вершины и по лемме~\ref{l:sepset} граф $D\backslash\{d\}$ разделим, противоречие. Предположим теперь, что у нас есть два отделимых множества $\{d,e\}$и $\{d,b\}$. Изолируем вершину $d$. Тогда по следствию~\ref{c:rn} вершина $e$ будет соединена со всеми остальными вершинами (за исключением вершины  $d$) ребрами одинакового веса. Аналогично с вершиной $b$. Но тогда, если удалить вершину $d$, множество $\{e,b\}$ будет отделимо по следствию~\ref{c:rd}, что противоречит неразделимости $D\backslash\{d\}$.

Пусть теперь $W=\{d,e\}$ --- единственное нетривиальное отделимое множество, содержащее $d$. Изолируем вершину $e$. Тогда по следствию~\ref{c:rn} вершина $d$ соединена со всеми вершинами из $D\backslash\{d,e\}$ ребрами одного веса. Предположим от противного, что граф $D\backslash\{d\}$ разделим; обозначим через $U$ его нетривиальное отделимое множество, не содержащее вершину $e$. Так как вершина $d$ соединена со всеми вершинами $U$ ребрами одного веса, то по следствию~\ref{c:rd} множество $U$ отделимо в $D$. Поскольку $V\backslash U$  содержит $d$ и несовпадает с $W$, получаем противоречие с единственностью $W$. 
%Получили, что $W=\{d,e\}$. Изолируем вершину $d$. По следствию 2 вершина e соединена со всеми вершинами графа $G\backslash\{d,e\}$ ребрами одинакового веса $i$. Граф $G\backslash\{e\}$ получается свитчингом из графа $G\backslash\{d\}$, если в качестве аддитивного графа взять граф, в котором вершина соединена со всеми остальными ребром веса $i$, и других ребер нет. Следовательно, он также неразделим.
\end{proof}
\begin{lemma}\label{l:nss}
Любой разделимый граф порядка $n \ge 5$ имеет либо $0$, 
либо $2$ неразделимых подграфа порядка $n-1$.
\end{lemma}
\begin{proof}
Пусть граф $G$ имеет неразделимый подграф $G\backslash\{a\}$ для некоторой вершины $a$.
Тогда по лемме~\ref{l:ns} отделимым множеством графа $G$ 
будет пара $\{a,b\}$ для некоторой вершины $b$.
Более того, других нетривиальных отделимых множеств, 
кроме этой пары и ее дополнения, в графе $G$
нет, откуда следует что граф $G\backslash\{b\}$ также неразделим. 
Удаление же из графа $G$ любой вершины $c$, отличной от $a$ и $b$,
приводит к разделимому подграфу, поскольку $\{a,b\}$
остается отделимым в $G\backslash\{c\}$.
\end{proof}

%1111111111111111111111111111111111111111111111111111111111111
\begin{predl}\label{l:allsep}
Пусть в графе $G$ порядка $n\ge 5$ 
каждый подграф порядка $4$ или $5$ разделим. 
Тогда граф $G$ разделим.
\end{predl}
\begin{proof}
Без потери общности мы можем считать, 
что граф содержит изолированную вершину $o$. 
По следствию~\ref{c:rn} 
в графе $G\backslash \{o\}$ 
не существует трех вершин таких, 
что веса ребер между этими вершинами попарно различны.  
Выберем в этом графе три вершины $a$, $b$, $c$, 
соединенные между собой ребрами двух различных весов 
(если таких не существует, 
то в подграфе $G\backslash\{o\}$ 
все ребра одного веса, 
и граф $G$ разделим): 
$E(a,b)=i\ne E(a,c)=E(b,c)=j$. 
Определим три множества $U$, $V$, $W$:

$U$ --- множество вершин, 
соединенных с вершинами 
$a$ и $b$ ребрами веса $i$;

$V$ --- множество вершин, 
$a$ и $b$ ребрами одинакового веса, 
отличного от $i$;

$W$ --- множество вершин, 
соединенных с вершинами 
$a$ и $b$ ребрами разного веса.

Для произвольных вершин $v$ из $V$ 
и $w$ из $W$ рассмотрим подграф, 
порождаемый множеством вершин $\{o,a,b,v,w\}$. 
Пусть ребра $\{v,a\}$ и $\{v,b\}$ 
имеют вес $l\ne i$. 
Одно из ребер $\{w,a\}$ и $\{w,b\}$ 
имеет вес $i$, скажем $\{w,a\}$.
Тогда ребро $\{w,b\}$ 
будет иметь вес $k\ne i$. 
По условию леммы 
рассматриваемый подграф 
на пяти вершинах разделим, 
а значит, имеет отделимую пару вершин.
Как легко убедиться, 
применяя следствие~\ref{c:rn}, 
такой парой может быть только 
$\{o,v\}$ или $\{w,b\}$, 
см. рис. \ref{f:oabvw}. 
В обоих случаях ребро $\{v,w\}$ 
должно иметь вес $l$.
В силу произвольности вершин $v$ и $w$ 
получаем, что любая вершина 
$v$ из $V$ соединена 
со всеми вершинами из $W\cup\{a,b\}$ 
ребрами одного веса, 
который обозначим через $l_v$.
\begin{figure}[ht]
\begin{minipage}[h]{0.3\linewidth}
\begin{center}
\begin{tikzpicture} 
    \fill (-0.0,0) circle (2pt) node [above] {$o$};
		\fill (1,-1) circle (2pt) node [left] {$a$};
		\fill (1, 1) circle (2pt) node [left] {$b$};
		\fill (3,-1) circle (2pt) node [right] {$v$};
		\fill (3, 1) circle (2pt) node [right] {$w$};
		\draw (1,-1) -- node [fill=white] {$i$} (1,1) 
		             -- node [near start,fill=white] {$l$} (3,-1) 
		             -- node [fill=white] {$\mathbf{?}$} (3, 1) 
		             -- node [near start,fill=white] {$i$} (1,-1) 
                     -- node [fill=white] {$l$} (3,-1);
		\draw (1, 1) -- node [fill=white] {$k$} (3, 1);
\end{tikzpicture}
\caption{$k\ne i$, $l\ne i$}
\label{f:oabvw}
\end{center}\end{minipage}
\hfill
\begin{minipage}[h]{0.3\linewidth}
\begin{tikzpicture} 
    \fill (-0.0,0) circle (2pt) node [above] {$o$};
		\fill (1,-1) circle (2pt) node [left] {$a$};
		\fill (1, 1) circle (2pt) node [left] {$w$};
		\fill (3,-1) circle (2pt) node [right] {$v$};
		\fill (3, 1) circle (2pt) node [right] {$u$};
		\draw (1,-1) -- node [fill=white] {$i$} (1,1) 
		             -- node [near start,fill=white] {$l_v$} (3,-1) 
		             -- node [fill=white] {$\mathbf{?}$} (3, 1) 
		             -- node [near start,fill=white] {$i$} (1,-1) 
                     -- node [fill=white] {$l_v$} (3,-1);
		\draw (1, 1) -- node [fill=white] {$k$} (3, 1);
\end{tikzpicture}
\caption{$k\ne i$, $l_v\ne i$}
\label{f:oauwv}
\end{minipage}
\hfill
\begin{minipage}[h]{0.3\linewidth}
\begin{tikzpicture} 
    \fill (-0.0,0) circle (2pt) node [above] {$o$};
		\fill (1,-1) circle (2pt) node [left] {$b$};
		\fill (1, 1) circle (2pt) node [left] {$w$};
		\fill (3,-1) circle (2pt) node [right] {$y$};
		\fill (3, 1) circle (2pt) node [right] {$u$};
		\draw (1,-1) -- node [fill=white] {$k$} (1,1) 
		             -- node [near start,fill=white] {$i$} (3,-1) 
		             -- node [fill=white] {$\mathbf{?}$} (3, 1) 
		             -- node [near start,fill=white] {$i$} (1,-1) 
                     -- node [fill=white] {$i$} (3,-1);
		\draw (1, 1) -- node [fill=white] {$k$} (3, 1);
\end{tikzpicture}
\caption{$k\ne i$}
\label{f:obuwy}
\end{minipage}
\end{figure}

Далее, множество $U$ разделим 
на два множества $U_1$, $U_2$: 
$U_1$ --- множество вершин, 
которые соединены со всеми вершинами 
из $W\cup\{a,b\}$ ребрами веса $i$; 
$U_2$ --- остальные вершины. 
Заметим, что если $U_2$ пусто, 
то из доказанного уже следует отделимость 
множества $V\cup U_1 \cup \{o\}$.

Рассмотрим  
произвольную вершину
$u$ из $U_2$. 
По определению множества $U_2$
найдется вершина $w$ из $W$ 
такая, что $E(u,w)\ne i$.
Без потери общности 
мы можем считать, что
$E(w,a)=i\ne E(w,b)$.
Для произвольной вершины $v$ из
$V$ рассмотрим подграф, порожденный
пятеркой $\{o,a,u,w,v\}$, 
см. рис. \ref{f:oauwv}.
Отделимой парой в этом подграфе
может быть только $\{o,v\}$
или $\{u,w\}$.
В обоих случаях $E(v,u)=l_v$.

При тех же условиях на $u$ и $v$
рассмотрим произвольную вершину 
$y$ из $U_1$. 
В графе, порожденном
пятеркой $\{o,b,u,w,y\}$, 
см. рис. \ref{f:obuwy},
отделимой парой может быть только
$\{o,y\}$ или $\{u,b\}$.
В обоих случаях $E(v,u)=i$.

Таким образом, 
каждая вершина $v$ из $V$ 
соединена со всеми вершинами из 
$\{a,b\}\cup W\cup U_2$ 
ребрами веса $l_v$, а
каждая вершина из $U_1$ 
соединена со всеми вершинами из 
$\{a,b\}\cup W\cup U_2$ 
ребрами веса $i$. 
По следствию~\ref{c:rd} 
множество $V\cup U_1 \cup \{o\}$ 
отделимо и граф $G$ разделим.
\end{proof}

%222222222222222222222222222222222222222222222222222222222222
\begin{predl}\label{t:2rs}
Пусть $G_K$ --- неразделимый подграф 
порядка $\chi$, $4 \le \chi< n-2$,
графа $G$ порядка $n$. 
И пусть все подграфы графа $G$ 
порядков $\chi+1$ и $\chi+2$, 
содержащие $G_K$ в качестве подграфа, 
разделимы. 
Тогда граф $G$ разделим.
\end{predl}
\begin{proof}
Возьмем произвольную вершину 
$a$ из $V\backslash K$. 
Тогда граф $G_{K\cup\{a\}}$ 
разделим, в то время как 
граф $G_K$  неразделим. 
По лемме~\ref{l:ns} 
в графе $G_{K\cup\{a\}}$ 
будет только одно 
(с точностью до дополнения) 
нетривиальное отделимое 
множество вершин $\{a,c\}$, 
для некоторой $c$ из $K$. 
Изолируем вершину $c$. 
Тогда по следствию~\ref{c:rn} 
вершина $a$ соединена 
со всеми вершинами из $K\backslash\{c\}$ 
ребрами одинакового веса. 
Для доказательства разделимости графа $G$ 
нам необходимо найти 
нетривиальное отделимое множество. 
Для этого в множестве $V\backslash K$ 
выделим подмножество 
$A$ вершин, которые соединены 
со всеми вершинами множества $K\backslash\{c\}$ 
ребрами одинакового веса, 
в частности, $a\in A$. 
В оставшейся части доказательства 
мы покажем, что 
множество $A \cup \{c\}$ отделимо в $G$.
Остаток $V\backslash K\backslash A$ 
обозначим через $B$. 

(*) \emph{Для произвольной вершины $d$ 
из $V\backslash (K\cup\{a\})$
докажем, что либо  $d \in A$, 
либо вершина $a$ соединена с $d$ 
ребром того же веса,
что и с вершинами из $K\backslash\{c\}$}. 
Рассмотрим граф $G_{K\cup\{a,d\}}$. 
По условию он разделим, 
поэтому вершина $d$ лежит 
в некотором нетривиальном 
отделимом в $G_{K\cup\{a,d\}}$ множестве $W$. 
Величина $m = |W \cap K|$ 
может принимать значения $3$ значения: 
$0$, $1$ или $\chi-1$
(иначе $W \backslash \{a,d\}$ 
--- нетривиальное отделимое множество в $G_K$, 
что противоречит неразделимости этого графа). 
Рассмотрим эти случаи.  

\begin{itemize}
\item
{$m=0$.}
В этом случае $\{a,d\}$ 
отделимо в графе $G_{K\cup\{a,d\}}$. 
По следствию~\ref{c:rn} 
любая вершина из графа $K$ 
соединена с этой парой вершин 
ребрами одного и того же веса.
Следовательно, вершина $d$ 
соединена со всеми вершинами множества 
$K\backslash\{c\}$
так же, как и вершина $a$, 
то есть ребрами одинакового веса. 
Имеем $d\in A$. 

\item
{$m=1$.}
В этом случае $W$ содержит 
некоторую вершину $e$ из $K$. 

Если $a$ принадлежит $W$, то $e=c$,
иначе в графе  $G_{K\cup\{a\}}$ 
будет больше одного отделимого множества, 
что противоречит лемме~\ref{l:ns}
(отметим, что для применения этой леммы 
 мы используем условие $\chi \ge 4$).
Таким образом, $W = \{a,c,d\}$, 
и легко видеть, что $d\in A$.

Если $a$ не принадлежит $W$, 
то $W=\{d,e\}$. 
Если $e=c$, то очевидно $d\in A$. 
Если $e\in K\backslash\{c\}$, 
то $E(a,d)=E(a,e)$, 
то есть вершина $a$ соединена с $d$ 
ребром того же веса,
что и с вершинами из $K\backslash\{c\}$.

\item
{$m=\chi{-}1$.} 
Отделимое дополнение до $W$ 
в графе $G_{K\cup\{a,d\}}$ 
есть пара $\{a,e\}$, 
для некоторого $e$ из $K$. 
Как и в предыдущем случае, 
мы видим, что $e=c$.
Следовательно,
$a$ соединена с $d$ 
ребром того же веса, что и 
со всеми вершинами из $K\backslash\{c\}$. 
\end{itemize}

Утверждение (*) доказано.

  Покажем теперь, 
что каждая вершина из $A$ 
соединена со всеми вершинами из $B$ 
ребрами одинакового веса. 
  Для вершины $a$ это верно по построению. 
  Но любая другая вершина $f$ из $A$ 
соединена со всеми вершинами 
графа $K\backslash\{c\}$ 
ребрами одинакового веса 
и образует с вершиной $c$ 
отделимое множество 
в графе $K\cup\{f\}$. 
  Поэтому для нее мы можем 
повторить те же рассуждения, 
что и для вершины $a$. 
  А так как любая вершина $g$ 
из $B$ по построению не соединена 
со всеми вершинами 
графа $K\backslash\{c\}$ 
ребрами одинакового веса, 
то вершина $f$ соединена 
с вершиной $g$ так же 
(то есть ребром того же веса), 
как и с вершинами графа $K\backslash\{c\}$. 
  Поэтому по следствию\ref{c:rd} 
множество $A\cup\{c\}$ 
является отделимым в графе $G$, 
и граф $G$ разделим.
\end{proof}

%000000000000000000000000000000000000000000000000
\begin{corollary}\label{c:2rs}
Если все подграфы 
порядков $n-1$ и $n-2$ 
графа $G$ порядка $n$ разделимы, 
то и сам граф $G$ разделим.
\end{corollary}
\begin{proof}
Пусть $\chi$ --- максимальный порядок 
собственного неразделимого подграфа. 
По условию $\chi<n-2$. 
Если $\chi=3$, 
то граф $G$ разделим 
по предложению~\ref{l:allsep}. 
Если $\chi>3$ 
то граф $G$ разделим 
по предложению~\ref{t:2rs}. 
\end{proof}
Чтобы сформулировать основную теорему,
при четном $q$
определим семейство графов $G_{n,\gamma}$,
$\gamma\in\{0,\ldots,q-1\}$, 
$n=2k+1 \ge 5$.
Множество вершин графа ---
$\{a_0,a_1,\ldots,a_k,b_1,\ldots,b_k\}$,
вершина $a_0$ изолирована, 
веса остальных ребер 
определяются следующим образом:
\begin{itemize}
\item 
для любых $l$, $m$ 
ребро $\{a_l,b_m\}$ 
имеет вес $\gamma$, если $l<m$, 
и вес $\gamma+q/2 (\mod q)$, 
если $l\ge m$;
\item 
для любых различных $l$, $m$ 
ребра $\{a_l,a_m\}$ и $\{b_l,b_m\}$ 
имеют вес $\gamma$.  
\end{itemize}

\begin{figure}[ht]

\begin{tikzpicture}  
    \draw(1,0);    
    \fill (8,4) circle (2pt) node [right] {$b_{1}$};
		\fill (6,4) circle (2pt) node [left] {$a_{1}$};
		\fill (8,3) circle (2pt) node [right] {$b_{2}$};
		\fill (6,3) circle (2pt) node [left] {$a_{2}$};
		\fill (8,2) circle (2pt) node [right] {$b_{3}$};
		\fill (6,2) circle (2pt) node [left] {$a_{3}$};
		\fill (8,0) circle (2pt) node [right] {$b_{\frac{n-1}2}$};
		\fill (6,0) circle (2pt) node [left] {$a_{\frac{n-1}2}$};
		
		\fill (10.5,4.5) circle (2pt) node [left] {$a_{0}$};
		
		\draw (6,0) -- (8,4);
		\draw (6,0) -- (8,3);
		\draw (6,0) -- (8,2);
		\draw (6,0) -- (8,0);
		 	 
		\draw (6,2) -- (8,4);
		\draw (6,2) -- (8,3);
		\draw (6,2) -- (8,2);
		
		\draw (6,3) -- (8,3);
		\draw (6,3) -- (8,4);
		
		\draw (6,4) -- (8,4);
		
	  \draw (6,0.5) node {$.$} ;
		\draw (6,1)  node {$.$} ;
		\draw (6,1.5)  node {$.$} ;
		\draw (8,0.5)  node {$.$} ;
		\draw (8,1)  node {$.$} ;
		\draw (8,1.5)  node {$.$} ;
	 \end{tikzpicture}
	\caption{Граф $G_{n,0}$}
	\label{f:gn}
	\end{figure}
	
%!!!!!!!!!!!!!!!!!!!!!!!!!!!!!!!!!!!!!!!!!!!!!!!!!!!!!!!!!!!!!!!!!!!!!!!
\begin{theorem}\label{t:main}
Если при удалении любой вершины 
из графа $G$ порядка $n$ 
всегда получается 
разделимый подграф графа $G$, 
то либо граф $G$ разделим, 
либо $q$ четно, $n$ нечетно, 
и граф $G$ изоморфен некоторому 
свитчингу графа $G_{n,\gamma}$, 
$\gamma\in\{0,\ldots,q-1\}$.
\end{theorem}

\begin{proof}
Пусть граф $G$ 
имеет порядок $n$ 
произвольной четности. 
Попробуем охарактеризовать 
все неразделимые графы 
$G$ порядка $n$, 
у которых все подграфы 
порядка $n-1$ разделимы.
Итак, пусть дан граф $G$ 
порядка $n\ge 5$, 
обладающий свойствами:  
\begin{enumerate}                                       
	\item[(i)] 
	$G$ --- неразделимый граф.
	\item[(ii)] 
	Для любой вершины $a$ из $G$
	граф $G\backslash \{a\}$
	разделимый.
\end{enumerate}

Будем говорить, 
что пара вершин графа $G$ 
обладает свойством $\{x,y\}_{*}$, 
если существует вершина $z$ такая, что
$\{x,y\}$ --- единственное 
(с точностью до дополнения) 
нетривиальное отделимое множество
в графе $G\backslash\{z\}$.
Если такая вершина $z$ известна, 
будем также обозначать 
это свойство через $\{x,y\}_z$.
По лемме~\ref{l:ns} 
($D=G\backslash\{z\}$) 
из свойства  $\{x,y\}_z$ следует, 
что графы $G\backslash\{z,x\}$ 
и $G\backslash\{z,y\}$ неразделимы.

Пусть $P$ --- 
множество всех пар $\{x,y\}$,
обладающих свойством $\{x,y\}_{*}$. 
Его можно построить, 
рассмотрев все подграфы 
порядка $n-1$:
если некоторая пара вершин 
является единственным 
отделимым множеством 
в одном из них, 
то эта пара принадлежит $P$.
 
% Мы хотим доказать, 
% что для любой вершины $z$ графа $G$ 
% существует пара вершин $\{x,y\}$, 
% такая, что выполняется 
% $\{x,y\}_{z}$. 
% [???? это ли мы хотим доказать??? 
% уместна ли здесь эта фраза????]

Рассмотрим несколько 
свойств множества $P$ 
и входящих в него пар. 

(iii) \emph{Множество $P$ не пусто.}
Это следует из следствия~\ref{c:2rs} 
и свойств (i) и (ii). 

(iv) \emph{Если для некоторой 
пары вершин $\{a,b\}$ 
выполняются соотношения 
$\{a,b\}_{c_1}$ 
и $\{a,b\}_{c_2}$, 
то $c_1=c_2$.}
  Действительно, 
пусть $\{a,b\}$ --- 
 отделимое множество
в графах $G\backslash\{c_1\}$ 
и $G\backslash\{c_2\}$. 
  Изолируем некоторую вершину 
$o$, отличную от
$a$, $b$, $c_1$, $c_2$. 
  Все вершины графа
$G\backslash\{c_1\}$ 
соединены с парой $\{a,b\}$ 
ребрами одинакового веса 
по следствию~\ref{c:rn};
то же верно для графа
$G\backslash\{c_2\}$. 
  По следствию~\ref{c:rd} пара 
$\{a,b\}$ отделима в графе  $G$, 
что противоречит свойству (i). 

(v) \emph{Любая вершина $a$ графа $G$
встречается ровно 
в двух парах множества $P$, 
либо не встречается совсем.} 
  Действительно, из
соотношения $\{a,b\}_c$
следует неразделимость графа
$G\backslash\{a,c\}$.
С другой стороны, из
неразделимости графа
$G\backslash\{a,c\}$
по лемме~\ref{l:ns} следует $\{a,b\}_c$ 
для некоторого $b$,
поскольку граф $G\backslash\{c\}$
разделим (свойство (i)).
По лемме~\ref{l:nss} 
в графе $G\backslash\{a\}$
ровно $2$ неразделимых подграфа:
$G\backslash\{a,c\}$ и 
$G\backslash\{a,c'\}$ 
для некоторых $c$ и $c'$.
Отсюда следует, что 
$\{a,b\}_c$ и $\{a,b'\}_{c'}$ 
для некоторых $b$ и $b'$,
и другой пары из $P$, 
содержащей $a$, нет.

(vi) \emph{Если $\{a,b\}_c$, 
то $\{c,d\}_a$ и $\{c,e\}_b$ 
для некоторых $d$ и $e$, 
причем $a\ne e\ne d\ne b$}.
Действительно, 
граф $G\backslash\{a,c\}$
неразделим, а граф 
$G\backslash\{a\}$ разделим,
следовательно, 
для некоторой вершины $d$ 
пара $\{c,d\}$ отделима 
в графе $G\backslash\{a\}$. 
Причем $d\ne b$, 
иначе все вершины множества
$V\backslash\{a,b,c\}$ 
(считая одну из них изолированной)
соединены с парами вершин 
$\{a,b\}$ и $\{b,c\}$ 
ребрами одинакового веса и,
следовательно, 
множество $\{a,b,c\}$ отделимо, 
что противоречит свойству (i).
Аналогично,
в графе $G\backslash\{b\}$ 
отделимое множество $\{c,e\}$ 
для некоторой вершины $e$,
отличной от $a$.
Из (iv) следует, что
$e\ne d$. 

% 5) \emph{Если пары 
% $\{c,d\}$ и $\{c,e\}$
% принадлежат $P$, 
% выполняется соотношение
% $\{a,b\}_c$, 
% и имеется изолированная вершина,
% отличная от $a$, $b$, $c$, $d$ и $e$,
% то ребра $\{c,d\}$, $\{c,e\}$  
% и $\{d,e\}$ имеют одинаковый вес}. 
% По пунктам 4 и 3 
% выполнены соотношения
% $\{c,d\}_a$ и $\{c,e\}_b$ 
% или $\{c,d\}_b$ и $\{c,e\}_a$. 
% Так как $\{c,e\}$ отделимо в
% $G\backslash\{b\}$ или в
% $G\backslash\{a\}$, 
% ребра $\{c,d\}$ и $\{d,e\}$ 
% имеют одинаковый вес. 
% Аналогично, $E(c,e)=E(e,d)$.

Перебирая пары из $P$ 
приведенным ниже алгоритмом,
построим
два множества вершин $A$, $B$. 
Выберем произвольную пару вершин
$\{a_0,a_1\}$ из $P$. 
Изолируем вершину $a_0$,
поместим $a_1$ в $A$.
Вершину $b_1$ такую, 
что  $\{a_0,a_1\}_{b_1}$, 
поместим в множество $B$. 
Тогда, согласно (vi), 
в графе $G\backslash \{a_1\}$  
отделимым множеством будет  
$\{b_1,b_2\}$, 
для некоторой вершины $b_2$. 
Помещаем вершину $b_2$ 
в множество $B$. 
Далее, 
в графе $G\backslash \{b_2\}$  
отделимым множеством 
будет пара $\{a_1,a_2\}$. 
Далее продолжаем аналогично, 
до тех пор, 
пока не окажется, что
очередная вершина 
уже была рассмотрена ранее,
то есть принадлежит 
$\{a_0\}\cup A\cup B$. 
Так как каждая рассматриваемая
вершина содержится 
ровно в двух парах из $P$,
этой вершиной может быть только $a_0$,
то есть, для некоторого $k$ верно 
либо $\{a_{k-1},a_0\}_{b_k}$
(в этом случае также выполнено 
$\{b_{k},b_1\}_{a_0}$, 
согласно (v)),
либо $\{b_{k},a_0\}_{a_k}$
(и $\{a_{k},b_1\}_{a_0}$).

Покажем, что первый 
из этих двух случаев
приводит к противоречию.
Выполнены следующие соотношения: 
$\{b_2,b_3\}_{a_2}$, \ldots, 
$\{b_{k-1},b_{k}\}_{a_{k-1}}$,
следовательно ребра 
$\{b_1,b_2\}$, \ldots, $\{b_1,b_{k}\}$
имеют одинаковый вес.
Аналогично, ребра 
$\{b_{k},b_{k-1}\}$, \ldots, $\{b_{k},b_{1}\}$
имеют один, 
тот же самый, вес.
В графе $G\backslash\{a_0\}$ 
отделимым множеством 
будет пара $\{b_1,b_k\}$.
Поскольку 
$E(b_1,b)=E(b_k,b)$,
при $b=b_2$, то же равенство
верно для любого $b$ из 
$V\backslash \{a_0,b_1,b_k\}$.
Следовательно, 
пара $\{b_1,b_{k}\}$
отделима в $G$, 
что противоречит 
свойству (i).

% $G\backslash\{a_0\}$ соединена с парой $\{b_1,b_k\}$ ребрами одного веса, либо любая другая вершина из графа $G\backslash\{a_0\}$ соединена с парой $\{b_1,b_k\}$  ребрами разного веса В превом случае $\{b_1,b_{k}\}$ отделимое множество в $G$. Во втором случае получаем противоречие с тем, что ребра $\{b_1,b_2\}$ и $\{b_{k},b_2\}$ имеют одинаковый вес.

Рассмотрим второй случай. 
Докажем, что множество 
$\{a_0\} \cup A\cup B$ 
содержит все вершины графа $G$. 
Предположим, что это не так. 
Возьмем произвольную вершину $a$ 
не из $\{a_0\} \cup A\cup B$.
Выполнены следующие соотношения: 
$\{b_1,b_2\}_{a_1}$, \ldots, 
$\{b_{k-1},b_{k}\}_{a_{k-1}}$,
следовательно ребра 
$\{a,b_1\}$, \ldots, $\{a,b_{k}\}$
имеют одинаковый вес.
Аналогично, из соотношений
$\{a_1,a_2\}_{b_2}$, \ldots, $\{a_{k-1},a_{k}\}_{b_{k}}$ ребра 
$\{a,a_1\}$, \ldots, $\{a,a_{k}\}$
имеют один вес. 
Из соотношений $\{a_0,a_1\}_{b_1}$ и $\{b_{k},a_0\}_{a_{k}}$
получаем, что $E(a_1,a)=E(a_1,b_{k})=E(b_{k},a)$
Получается, что вершина $a$ 
соединена со всеми вершинами 
множества  $A\cup B$ 
ребрами одинакового веса. 
В силу произвольности вершины $a$, 
это верно для всех остальных вершин, 
не лежащих в $A\cup B\cup a_0$. 
А следовательно, 
все эти вершины 
вместе с вершиной $a_0$ 
по следствию~\ref{c:rd} 
образуют отделимое множество 
в графе $G$, 
что противоречит его неразделимости.

Итак, $n=2k+1$  нечетно, 
множество вершин состоит 
из изолированной вершины $a_0$ 
и подмножеств $A=\{a_1,\ldots,a_k\}$ 
и $B=\{b_1,\ldots,b_k\}$. 
Вершины связаны следующими соотношениями:
$$
\{a_0,a_1\}_{b_1}
%\{a_1,a_2\}_{b_2}
\ldots
\{a_{k-1},a_{k}\}_{b_k},
\{a_k,b_1\}_{a_0},
\{b_1,b_2\}_{a_1}
\ldots
\{b_{k-1},b_k\}_{a_{k-1}},
\{b_{k},a_0\}_{a_{k}}
$$

Установим серию равенств 
для весов между ребрами графа.
Для любого $i\in\{1,\ldots ,k\}$ 
верно следующее:
\begin{eqnarray}
E(a_i,a_j)&=&E(a_i,a_{j+1}), \qquad 
\forall j\in\{i+1,\ldots ,k-1\};
  \label{eq:aa+} \\
E(a_i,a_l)&=&E(a_i,a_{l-1}), \qquad 
\forall l\in\{2,\ldots ,i-1\};
  \label{eq:aa-} \\
E(b_i,b_j)&=&E(b_i,b_{j+1}), \qquad 
\forall j\in\{i+1,\ldots ,k-1\};
  \label{eq:bb+} \\
E(b_i,b_l)&=&E(b_i,b_{l-1}), \qquad 
\forall l\in\{2,\ldots ,i-1\};
  \label{eq:bb-} \\
E(a_i,b_j)&=&E(a_i,b_{j+1}), \qquad 
\forall j\in\{i+1,\ldots ,k-1\};
  \label{eq:ab+} \\
E(a_i,b_l)&=&E(a_i,b_{l-1}), \qquad 
\forall l\in\{2,\ldots ,i\};
  \label{eq:ab-} \\
E(b_i,a_j)&=&E(b_i,a_{j+1}), \qquad 
\forall j\in\{i,\ldots ,k-1\};
  \label{eq:ba+} \\
E(b_i,a_l)&=&E(b_i,a_{l-1}), \qquad 
\forall l\in\{2,\ldots ,i-1\};
  \label{eq:ba-} \\
E(a_1,a_k)&=&E(a_1,b_k) \ne E(a_1,b_1);
  \label{eq:a1ak} \\
E(b_k,a_1)&=&E(b_k,b_1) \ne E(b_k,a_k).
  \label{eq:bka1}
\end{eqnarray}
Действительно, соотношения 
(\ref{eq:aa+}), (\ref{eq:aa-}), 
(\ref{eq:ba+}), (\ref{eq:ba-})
следуют из 
$\{a_{m-1},a_{m}\}_{b_m}$, 
$m\in\{2,\ldots ,k\}$: 
в графе $G\backslash \{b_m\}$
дополнение до отделимой 
пары $\{a_{m-1},a_{m}\}$
содержит изолированную вершину $a_0$,
а значит, по следствию~\ref{c:rn}, 
для любой вершины $c$ этого
дополнения верно $E(c,a_{m-1})=E(c,a_m)$.
Аналогично, соотношения 
(\ref{eq:bb+})--(\ref{eq:ab-})
следуют из 
$\{b_{m-1},b_{m}\}_{a_{m-1}}$. 
Соотношение (\ref{eq:a1ak}) 
(аналогично (\ref{eq:bka1})) 
следует по следствию~\ref{c:rn}
из $\{a_0,a_1\}_{b_1}$
(соответственно $\{b_k,a_0\}_{a_k}$),
поскольку в графе $G\backslash \{b_1\}$
пара $\{a_0,a_1\}$ отделимая, 
но в графе $G$ она уже не отделима.

Из равенств 
(\ref{eq:aa+}) и (\ref{eq:aa-}) 
(соответсвенно, 
(\ref{eq:bb+}) и (\ref{eq:bb-})) 
легко заключить, 
что все ребра, соединяющие
вершины из множества 
$A$ ($B$) между собой,
имеют один и тот же вес, 
скажем $\alpha$ ($\beta$).
Из равенств (\ref{eq:ab+}) 
и (\ref{eq:ba-}) следует, 
что все ребра $\{a_i,b_j\}$, 
где $1 \le i < j \le k$, 
имеют один и тот же вес, 
скажем $\gamma$.
Аналогично, из 
(\ref{eq:ab-}) и (\ref{eq:ba+})
вытекает, что все ребра 
$\{a_i,b_j\}$, 
где $1 \le j \le i \le k$,  
имеют один и тот же вес, 
скажем $\delta$.
Более того из 
(\ref{eq:a1ak}) и (\ref{eq:bka1})
мы видим, что 
$\alpha=\gamma=\beta\ne\delta$.

Остается показать, 
что $\delta=\gamma+q/2$.
В графе $G\backslash\{a_0\}$  
отделима пара $\{a_{k},b_1\}$. 
По лемме~\ref{l:sepset} 
и в соответствии с ее обозначениями, 
$\{a_{k}\}=W_{a}$, $\{b_1\}=W_{b}$, 
$A\backslash \{a_k\}=V_{a'}$, 
$B\backslash \{b_1\}=V_{b'}$, 
причем $a+a'=b+b'=\gamma$ 
и $a+b'=b+a'=\delta$.
Отсюда видно, что 
$2\gamma=2\delta$. 
Поскольку $\gamma\ne\delta$, 
имеем $\delta=\gamma+q/2$, 
где $q$ четно.
\end{proof}

Для полноты картины 
остается показать, 
что сам граф $G_{n,\gamma}$
действительно является исключением.

\begin{predl}\label{p:Gnc}
 Граф $G_{n,\gamma}$ 
 ($n$ нечетно, $q$ четно) 
 неразделим и все его подграфы 
 порядка $n-1$ разделимы.
\end{predl}
\begin{proof}
Предположим от противного, что 
$G_{n,\gamma}=(V,E)$ разделим. 
Пусть $W$ --- нетривиальное 
отделимое множество, 
содержащее $a_0$.
Вершины $a_k$ и $b_1$ 
не могут одновременно 
принадлежать $V\backslash W$ 
(иначе, согласно следствию~\ref{c:rn} 
в $W$ не может быть
никакой вершины, кроме $a_0$). 
Если $a_k\in W$, 
то либо $\{a_1,\ldots,a_{k-1}\}$,
либо $\{b_1,\ldots,b_k\}$ 
полностью содержится в $W$
(иначе имеем противоречие 
со следствием~\ref{c:rn}).
В первом случае $V\backslash W$ содержит
некоторые $b_i$, $b_j$, $i<j$, 
и вершина $a_i$ из $W$
противоречит следствию~\ref{c:rn}.
Во втором случае $V\backslash W$ содержит
некоторые $a_i$, $a_j$, $i<j$, 
и вершина $b_j$ из $W$
противоречит следствию~\ref{c:rn}.
Случай $b_1\in W$ аналогичен.
Полученное противоречие 
доказывает неразделимость 
графа $G_{n,\gamma}$.

Остается заметить,
что пара $\{a_k,b_1\}$ отделима 
в $G\backslash \{a_0\}$,
пара $\{b_{i},b_{i+1}\}$ отделима 
в $G\backslash \{a_i\}$, 
$i=1,\ldots ,k-1$,
пара $\{b_{k},a_0\}$ отделима 
в $G\backslash \{a_k\}$,
пара
$\{a_{i-1},a_i\}$ отделима 
в $G\backslash \{b_i\}$, 
$i=1,\ldots ,k$.
\end{proof}

%=============================================================================
\section{Функции и квазигруппы}\label{s:gbq}
В этом разделе мы кратко обсудим 
связь разделимости графов 
с аналогичным свойством 
для функций над $\mathbb{Z}_q$
и для $n$-арных квазигрупп.

Подграфам графа соответствуют так называемые ретракты 
$n$-арных квазигрупп и подфункции функций, 
и то и другое получается фиксацией некоторых аргументов. 
В терминах ретрактов и подфункций 
для $n$-арных  квазигрупп и булевых функций 
верны утверждения, аналогичные следствию~\ref{c:2rs} 
и предложению~\ref{p:Gnc}
(последнее известно для $n$-арных  квазигрупп только если порядок кратен $4$). 
Причем, учитывая, что по графу при помощи квадратичного многочлена 
можно построить функцию и затем $n$-арную квазигруппу порядка $q^2$, 
следствие~\ref{c:2rs} является, вообще говоря, 
следствием соответствующего утверждения для квазигрупп 
\cite{Kro:n-3,KroPot:4} 
(заметим, что приведенное в настоящей работе для полноты 
доказательство этого следствия 
на порядок проще доказательства для квазигрупп), 
а из предложения~\ref{p:Gnc}, наоборот, 
следует существование аналогичного примера 
среди $n$-арных квазигрупп порядка $4$ \cite{Kro:n-2}.

\subsection{Частичные функции}
Пусть $q$ --- простое число.
\emph{Частичной функцией} будем называть 
частичную функцию $\extshn{a}f:\Omega_a\to \{0,\ldots ,q-1\}$ 
над $\mathbb{Z}_q$, 
заданную на наборах значений аргументов, 
сумма которых равна некоторой константе $a$:
$$\Omega_a = \bigg\{(x_1,\ldots ,x_n,x_0)\in \{0,\ldots ,q-1\}^{n+1} \Big| \sum_{i=0}^n x_i = a\bigg\}.$$
Заметим, что частичную функцию от $n+1$ аргумента
можно интерпретировать 
как обычную (везде определенную) функцию 
от первых $n$ аргументов.
Частичную функцию 
$\extshn{a}f$ будем называть \emph{расширением}, 
или \emph{$a$-расширением} функции $f$, если
$$ \extshn{a}f(x_1,\ldots ,x_n,x_0) = f(x_1,\ldots ,x_n)$$
для любых $x_1,\ldots ,x_n$ и $x_0= a-\sum_{1}^n x_i$.
Далее обозначение
 $\extshn{a}f$ будет автоматически подразумевать,
 что $\extshn{a}f$ есть $a$-расширение функции $f$.

Пусть $W$ --- некоторое подмножество множества номеров аргументов $\{1,\ldots ,n,0\}$.
Частичную функцию $\extshn{a}f$ 
от $n+1$ аргумента 
назовем \emph{$W$-разделимой}, 
если 
$$ \extshn{a}f(x_1,\ldots ,x_n,x_0) = f'(x_W) + f'' (x_{U}) $$
для некоторых (везде определенных) функций $f'$ и $f''$ 
от $|W|$ и $|U|$ аргументов, соответственно, где
$U=\{1,\ldots ,n,0\}\backslash W$, и $x_C = (x_{c_1},\ldots ,x_{c_k})$
для $C=\{c_1,\ldots ,c_k\}$, $c_1<\ldots <c_k$.
Частичную функцию от $n+1$ аргумента
назовем \emph{разделимой}, если она 
$W$-разделима для некоторого $W$ мощности от $2$ до $n-1$ включительно
(ограничение на мощность достаточно естественно: 
в этом случае функции в разложении 
могут быть заданы 
меньшим числом 
значений в точках, 
чем сама частичная функция). 
\emph{Степенью} частичной функции 
назовем минимальную степень 
многочлена (над полем $\mathbb{Z}_q$), 
с помощью которого 
она может быть представлена.
Под термином <<квадратичный>> 
будем подразумевать 
<<степени не больше двух>>.
Графом квадратичного многочлена 
назовем граф на множестве номеров аргументов $\{1,\ldots ,n,0\}$,
у которого вес ребра, 
соединяющего две различные вершины, 
равен коэффициенту в многочлене 
при произведении соответствующих
переменных
(таким образом, две вершины 
смежны тогда и только тогда, 
когда произведение соответствующих 
переменных входит в многочлен).
\begin{lemma}\label{l:bol-swe}
Множество графов, 
соответствующих представлениям 
данной квадратичной частичной функции
в виде квадратичного многочлена, 
образует свитчинговый класс.
\end{lemma}
\begin{proof}
Любая (и в частности, квадратичная)
частичная функция $\extshn{a}f$ от $n+1$ аргументов, 
будучи функцией от первых $n$ своих аргументов,
единственным образом представима в виде
$$ \extshn{a}f(x_1,\ldots,x_{n},x_{0})=\pi(x_1,\ldots,x_{n}) ,$$
где $\pi$ --- многочлен 
(поскольку в поле $\mathbb{Z}_q$ 
имеет место тождество $x^q \equiv x$, 
степень каждой переменной в мономах такого многочлена ограничена $q-1$).

%Далее без потери общности считаем, что $a=0$.
Любой  многочлен $\rho$ от $n+1$ 
переменной $x_1,\ldots,x_{n},x_{0}$
однозначно представим в виде
$$ \tau(x_1,\ldots,x_{n})+(x_1+\dots+x_{n}+x_{0}-a)\lambda(x_1,\ldots,x_{n}),$$
где $\tau$ и $\lambda$ --- многочлены от $x_1,\ldots,x_{n}$, 
причем если многочлен $\rho$ квадратичный, 
то $\tau$ квадратичный и $\lambda$ аффинный (степени не больше $1$).
Поскольку $x_1+\dots+x_{n}+x_{0}=a$ 
везде на области определения частичной функции $ \extshn{a}f$,
многочлен $\tau$ совпадает у всех многочленов, 
представляющих $ \extshn{a}f$.
Легко убедиться, что добавление 
$(x_1+\dots+x_{n}+x_{0}-a)\lambda(x_1,\ldots,x_{n})$ 
с аффинным $\lambda$ приводит к свитчингу соответствующего графа.
Отсюда следует утверждение леммы.
\end{proof}

\begin{predl}\label{l:bol-sep}
Квадратичная частичная функция $\extshn{a}f$
является $W$-разделимой тогда и только тогда,
когда множество вершин $W$ отделимо в графах 
квадратичных многочленов, 
представляющих эту функцию.
\end{predl}

\begin{proof}
С учетом леммы~\ref{l:bol-swe},
$W$-разделимость функции, 
представимой квадратичным 
многочленом, следует по определению
из отделимости множества $W$ в соответствующем графе.

Для доказательства обратного, 
с учетом леммы~\ref{l:bol-swe}, 
достаточно показать,
что для разделимой квадратичной 
частичной функции $\extshn{a}f$ 
элементы ее некоторого разложения $f'$ и $f''$ 
из определения разделимости также квадратичны.
Пусть 
\begin{equation}\label{eq:f=gg}
\extshn{a}f(y_1,\ldots ,y_k,z_1,\ldots ,z_m)
   =g'(y_1,\ldots ,y_m)+g''(z_1,\ldots ,z_m)
\end{equation}
для любых значений аргументов, 
удовлетворяющих тождеству 
$y_1+\ldots +y_k+z_1+\ldots +z_m=a$
Подставляя нули 
в качестве значений 
некоторых аргументов, 
получаем
\begin{eqnarray}
 \extshn{a}f(y_1,\ldots ,y_k,-\sum_1^k y_i,0,\ldots ,0)
  &=&g'(y_1,\ldots ,y_m)+g''(-\sum_1^k y_i,0,\ldots ,0), \label{eq:y0}\\
 \extshn{a}f(-\sum_1^m z_i,0,\ldots ,0,z_1,\ldots ,z_m)
  &=&g'(-\sum_1^m z_i,0,\ldots ,0)+g''(z_1,\ldots ,z_m), \label{eq:z0}\\
 \extshn{a}f(t,0,\ldots ,0,-t,0,\ldots ,0)
  &=&g'(t,0,\ldots ,0)+g''(-t,0,\ldots ,0), \label{eq:0t0}
\end{eqnarray}
Из тождеств (\ref{eq:f=gg})--(\ref{eq:0t0}) 
легко получить разложение
\begin{eqnarray*}
 \extshn{a}f(y_1,\ldots ,y_k,z_1,\ldots ,z_m)
  &=&f(y_1,\ldots ,y_k,-\sum_1^k y_i,0,\ldots ,0) \\
&& +f(-\sum_1^m z_i,0,\ldots ,0,z_1,\ldots ,z_m) \\
&& -f(-\sum_1^m z_i,0,\ldots ,0,-\sum_1^k y_i,0,\ldots ,0).
\end{eqnarray*}
Поскольку 
$-\sum_1^m z_i=\sum_1^k y_i$, 
последнее слагаемое 
можно отнести как к первой,
так и ко второй части разложения.
В итоге мы имеем искомое разложение на 
квадратичные функции.
\end{proof}

Если вместо одного или более аргументов частичной функции 
подставить константы, мы получим \emph{подфункции} данной частичной функции.
Очевидно, подфункции сами являются частичными функциями. 
Например, если у частичной функции $\extshn{a}f$ зафиксировать
значение одного из аргументов значением $b$, 
то получится некоторая частичная функция $\extshn{a+b}g$.

Легко понять, что графы, 
соответствующие подфункции квадратичной частичной функции $\extshn{a}f$,
получаются из графов, соответствующих $\extshn{a}f$,
удалением вершин --- номеров фиксируемых аргументов.
Из предложения~\ref{l:bol-sep} и теоремы~\ref{t:main} мы можем заключить следующее.

\begin{corollary}\label{c:subfun}
 Пусть $q$ --- нечетное простое число. 
 Если из квадратичной частичной функции  
 $\extshn{a}f$ фиксацией любого аргумента
получается разделимая частичная функция, то и сама $\extshn{a}f$ разделима.

Другими словами, из любой неразделимой квадратичной частичной функции фиксацией некоторой переменной можно получить неразделимую подфункцию.
\end{corollary}

%--------------------------------------------------------------
\subsection{$n$-Арные квазигруппы}
Пусть $\Sigma$ --- некоторое множество. 
$n$-Арная операция 
$Q:\Sigma^n \to \Sigma$ 
называется
\emph{$n$-арной квазигруппой} 
(далее, также просто \emph{квазигруппой})
порядка $|\Sigma|$,
если в уравнении 
$x_0=Q(x_1,\ldots,x_n)$ 
значения любых $n$ переменных 
однозначно задают 
значение оставшейся переменной.
(Строго говоря, $n$-арной квазигруппой 
называется система $(\Sigma,Q)$ из носителя и операции,
наше определение --- общепринятое упрощение терминологии.)
Из определения следует, 
что $n$-арная квазигруппа 
обратима в каждой позиции,
в случае конечного порядка 
это свойство можно взять за определение.
Квазигруппу, обратную к $Q$ в $i$м месте, будем обозначать через $Q^{(i)}$.
Таким образом, для $n$-арной квазигруппы $Q$ соотношения $x_0=Q(x_1,\ldots,x_n)$ и 
$x_i=Q^{(i)}(x_1,\ldots,x_{i-1},x_0,x_{i+1},\ldots,x_n)$ 
эквивалентны.
Квазигруппу арнсти $n-k$, 
полученную из $n$-арной квазигруппы $Q$ 
или ее обращения  $Q^{(i)}$, 
$i\in\{1,\ldots ,n\}$, 
подстановкой констант 
на место некоторых $k$ аргументов, 
называют \emph{ретрактом} квазигруппы $Q$.

Пусть $W$ --- некоторый набор номеров аргументов 
$n$-арной квазигруппы $Q$. 
Квазигруппу $Q$ назовем \emph{$W$-разделимой}, если
$Q(\bar x) = G(\bar x',H(\bar x''))$
для некоторых $|W|$-арной квазигруппы $H$
и $(n-|W|+1)$-арной квазигруппы $G$, 
где $\bar x''$ --- набор переменных из $\bar x$ с номерами из $W$, $\bar x'$ состоит из всех остальных переменных $\bar x$.
Другими словами,
уравнение $x_0 = Q(\bar x)$ эквивалентно 
$G^{(n-|W|+1)}(\bar x',x_0)=H(\bar x'')$, поэтому $W$-разделимую
$n$-арную квазигруппу будем считать также 
$\{1,\ldots ,n,0\}\backslash W$-разделимой.
Квазигруппу $Q$ назовем \emph{разделимой}, 
если она $W$-разделима для некоторого $W$, 
где $2\le |W|\le n$
(это неравенство эквивалентно тому, что квазигруппы, 
входящие в соответствующую суперпозицию,
имеют меньшую арность).

Рассмотрим в качестве $\Sigma$ множество пар 
$\{ [a,b] | a,b\in \mathbb{Z}_q\}$.
Для произвольной функции 
$f:\mathbb{Z}_q^n \to \mathbb{Z}_q$ 
и константы $a\in \mathbb{Z}_q$ определим 
$n$-арную квазигруппу $Q_f$ порядка $q^2$:

$$ Q_{f,a}([x_1,y_1],\ldots ,[x_n,y_n]) 
  =  [-\sum_{i=1}^n x_i+a, 
  -\sum_{i=1}^n y_i+f(x_1,\ldots ,x_n)]. $$

\begin{predl}\label{l:f-to-q}
  $n$-Арная квазигруппа $Q_{f,a}$ 
является $W$-разделимой
тогда и только тогда, когда 
$W$-разделима частичная функция $\extshn{0}f$.
\end{predl}
\begin{proof}
Без потери общности положим $W=\{k,\ldots ,n\}$.
Пусть $\dot f$ является $W$-разделимой.
Тогда $f = f'(x_1,\ldots ,x_{k-1},-\sum_{1}^n x_i) + f''(x_k,\ldots ,x_n)$
для некоторых (везде определенных) функций $f'$ и $f''$.
Определим $g'(x_1,\ldots ,x_{k-1},t) = f'(x_1,\ldots ,x_{k-1},-x_1-\ldots -x_{k-1}-t)$.
Теперь нетрудно проверить, что
\begin{eqnarray*}
Q_{g'}\big([x_1,y_1],\ldots ,[x_{k-1},y_{k-1}], -Q_{f''}([x_k,y_k],\ldots ,[x_{n},y_{n}])\big)
=
\bigg[-\sum_{i=1}^{k-1} x_i-\sum_{i=k}^n x_i, \\ 
-\sum_{i=1}^{k-1} y_i
-\sum_{i=k}^n y_i + f''(x_k,\ldots ,x_n)
+ g'\Big( x_1,\ldots ,x_{k-1},\sum_{i=k}^n x_i\Big)\bigg]
=Q_f([x_1,y_1],\ldots ,[x_{n},y_{n}]).
\end{eqnarray*}
Таким образом, $n$-арная квазигруппа $Q_f$ также является $W$-разделимой.

Теперь предположим, что $Q_f$ является $W$-разделимой.
Рассмотрим функцию 
\begin{eqnarray*}
g (x_1,\ldots ,x_n) &= &f(x_1,\ldots ,x_{k-1},\sum_{k}^n x_i, 0,\ldots ,0) \\
+ f (0,\ldots ,0, x_k,\ldots ,x_n) &-& f(0,\ldots ,0, \sum_{k}^n x_i, 0,\ldots ,0).
\end{eqnarray*}
Расширение $\extshn{0}g$ этой функции $W$-разделимо. 
Действительно, для $x_1+\ldots +x_n+x_0=0$
\begin{eqnarray*}
\dot g (x_1,\ldots ,x_n,x_0)& =& f(x_1,\ldots ,x_{k-1},\sum_{0}^{k-1} x_i, 0,\ldots ,0) 
\\ && + (f (0,\ldots ,0, x_k,\ldots ,x_n) 
- f(0,\ldots ,0, \sum_{k}^n x_i, 0,\ldots ,0)).
\end{eqnarray*}
Согласно первой части доказательства, 
$n$-арная квазигруппа $Q_g$ также является $W$-разделимой.
Кроме того, 
\begin{eqnarray*}
Q_f(z_1,\ldots ,z_k,[0,0],\ldots ,[0,0]) = 
Q_g(z_1,\ldots ,z_k,[0,0],\ldots ,[0,0])\qquad\mbox{и}
\\
Q_f([0,0],\ldots ,[0,0],z_k,\ldots ,z_n) = 
Q_g([0,0],\ldots ,[0,0],z_k,\ldots ,z_n)
\end{eqnarray*}
для всех значений $z_1$, \ldots, $z_n$.
Согласно \cite[Предложение~12]{PotKro:asymp.ru}, 
две $W$-разделимые $n$-арные квазигруппы,
для которых указанные ретракты совпадают, 
сами обязаны совпадать.
Как следствие, совпадают функции $f$ и $g$, 
и $\extshn{0}f$ также является $W$-разделимой. 
\end{proof}

Нам осталось связать 
разделимость подфункций 
с разделимостью ретрактов. 
Для этого нам понадобятся 
следующие два утверждения.

\begin{lemma}\label{l:retr}
  Пусть $(n-1)$-арная квазигруппа $F$ 
получается из $Q_{f,a}$ 
подстановкой константы $[b,c]$ 
вместо $i$го аргумента. 
  Тогда $F=Q_{g,a+b}$, 
где функция $g$ 
получается из $f$ 
подстановкой константы $b$ 
на место $i$го аргумента
и прибавлением $c$ к значению.
\end{lemma}
Ясно, что прибавление константы к значению функции никак не влияет на ее разделимость.
Поэтому получаем, что разделимость ретракта $F$ эквивалентна разделимости соответствующей 
подфункции расширения функции $f$.
Однако, таким образом представимы не все $(n-1)$-арные ретракты. 
\begin{lemma}\label{l:obr}
Пусть функции $f$ и $g$ связаны тождеством 
$$\extshn{a}{f}(x_1,\ldots ,x_{i-1},x_i,x_{i+1},\ldots ,x_n,x_0) = 
\extshn{a}{g}(x_1,\ldots ,x_{i-1},x_0,x_{i+1},\ldots ,x_n,x_i)$$ 
для некоторого $i$ от $1$ до $n$. Тогда $n$-арные квазигруппы
$Q_{f,a}$ и $Q_{g,a}$ являются обращением друг друга в $i$м месте, то есть
$Q_{f,a}^{(i)}=Q_{g,a}$.
\end{lemma}
\begin{proof}
Для простоты обозначений, рассмотрим только случай $i=1$. 
По определению $Q_{f,a}$, уравнение 
$[x_0,y_0]=Q_{f,a}([x_1,y_1],\ldots ,[x_n,y_n])$ 
эквивалентно системе
\begin{eqnarray*}
 \begin{cases}
  -x_0-x_1-\ldots -x_n+a=0 \\
  -y_0-y_1-\ldots -y_n+\extshn{a}f(x_1,\ldots ,x_n,x_0)=0.
 \end{cases}
\end{eqnarray*}
По определению обращения квазигруппы в первом аргументе, 
той же системе эквивалентно и уравнение
$[x_1,y_1]=Q_{f,a}^{(1)}([x_0,y_0],[x_2,y_2],\ldots ,[x_n,y_n])$.
Но эта же система эквивалентна уравнению 
$$[x_1,y_1]=Q_{g,a}([x_0,y_0],[x_2,y_2],\ldots ,[x_n,y_n]).$$
\end{proof}
Таким образом, взятие обращения квазигруппы $Q_{f,a}$ 
соответствует перестановке двух аргументов
частичной функции $\extshn{a}{f}$. Как следствие,
разделимость ретрактов квазигруппы  $Q_{f,a}$ эквивалентна 
разделимости соответствующих подфункций 
частичной функции $\extshn{a}{f}$.

\begin{corollary}\label{c:retr}
 Пусть $q>2$ простое и 
 $f:\mathbb{Z}_q^n \to \mathbb{Z}_q$ 
 --- квадратичная функция.
Если все $(n-1)$-арные ретракты $n$-арной квазигруппы $Q_{f,a}$ порядка $q^2$ разделимы,
то и сама квазигруппа разделима.
\end{corollary}

%\bibliographystyle{plain}
%\bibliography{k}

\providecommand\href[2]{#2} \providecommand\url[1]{\href{#1}{#1}}
  \providecommand\bblmay{May} \providecommand\bbloct{October}
  \providecommand\bblsep{September} \def\DOI#1{{\small {DOI}:
  \href{http://dx.doi.org/#1}{#1}}}\def\DOIURL#1#2{{\small{DOI}:
  \href{http://dx.doi.org/#2}{#1}}}\providecommand\bbljun{June}

\end{document}